\documentclass[12pt]{amsart}
\usepackage{color,latexsym,amssymb,amsmath,amsthm}
\usepackage{comment}
\usepackage{graphicx}

\definecolor{orange}{rgb}{1,0.5,0}

\newtheorem{thm}{Theorem}
\newtheorem{coro}[thm]{Corollary}
\newtheorem{lemma}[thm]{Lemma}
\newtheorem{propo}[thm]{Proposition}
\theoremstyle{definition}
\newtheorem*{defn}{Definition}

\newcommand{\E}{{\rm E}}

\title{On index expectation and curvature for networks}
\author{Oliver Knill}
\date{Feb 20, 2012}
\address{
        Department of Mathematics \\
        Harvard University \\
        Cambridge, MA, 02138
        }


\begin{document}
\maketitle

\begin{abstract}
We prove that the expectation value of the 
index function $i_f(x)$ over a probability space of injective function $f$ on any 
finite simple graph $G=(V,E)$ is equal to the curvature $K(x)$ at the vertex $x$.
This result complements and links Gauss-Bonnet $\sum_{x \in V} K(x) = \chi(G)$ 
and Poincar\'e-Hopf $\sum_{x \in V} i_f(x) = \chi(G)$ which both hold for 
arbitrary finite simple graphs. 
\end{abstract}

\section{Introduction}

For a general finite simple graph $G=(V,E)$, the {\bf curvature} at a vertex $x$ is defined as the finite sum
$$  K(x) = \sum_{k=0}^{\infty} (-1)^k \frac{V_{k-1}(x)}{k+1} \; , $$
where $V_k(x)$ is the number of $K_{k+1}$ subgraphs in the sphere $S(x)$ at a vertex $x$ and $V_{-1}(x)=1$. 
With this curvature, the Gauss-Bonnet theorem \cite{cherngaussbonnet}
$$ \sum_{x \in V} K(x) = \chi(G) \;  $$
holds, where $\chi(G)=\sum_{k=0}^{\infty} (-1)^k v_k$ is the {\bf Euler characteristic} of the graph, and where
$v_k$ is the number of $K_{k+1}$ subgraphs of $G$. For example, if $G$ contains no tetrahedral subgraph 
$K_4$, then each sphere $S(x)$ lacks triangular subgraphs and $K(x) = 1-V_0(x)/2+V_1(x)/3$ and 
$\chi(G)=|V|-|E|+|T|$, where $T$ is the set of triangular subgraphs $K_3$ of $G$. 
For an injective function $f$ on the vertex set $V$, the {\bf index} at a vertex is defined as the integer
$$  i_f(x) = 1-\chi(S^-(x)) \; , $$
where $S^-(x) = \{ y \in S(x) \; | \; f(y)<f(x) \; \}$ 
is the {\bf exit set} of the unit sphere $S(x)$ with respect to the gradient field of $f$ and
where $\chi(H)$ is the Euler characteristic of a subgraph $H$ of $G$. 
The index $i_f(x)$ is a discrete version of the {\bf Brouwer index} for gradient vector fields and satisfies the discrete 
Poincar\'e-Hopf theorem \cite{poincarehopf}
$$ \sum_{x \in V} i_f(x) = \chi(G) \; , $$
a result which holds for arbitrary simple graphs. Poincar\'e-Hopf gives a fast way to compute the
Euler characteristic because the subgraphs $S^-(x)$ are in general small. 
This allows to compute $\chi(G)$ for random graphs with hundreds of vertices, where counting cliques would be hopeless.\\

Since Poincar\'e-Hopf works for all injective $f$, also 
the {\bf symmetric index} $j_f(x) = [i_f(x)  + i_{-f}(x)]/2$ 
satisfies $\sum_{x \in V} j_f(x) = \chi(G)$. For cyclic graphs $G$, $j_f(x)$ is zero everywhere and agrees with curvature $K(x)$.
For trees, the symmetric index satisfies $j_f(x) = 1-{\rm deg}(x)/2=1-|V_0(x)|/2$ which adds up to $1$ for connected trees.
For graphs in which every unit sphere is a cyclic graph like the icosahedron,
we have $\chi(S^+(x))=\chi(S^-(x))$ and $j_f=i_f$. The curvature $K(x)$ is then $1-|V_0(x)|/2+|V_1(x)|/3=1-|S(x)|/6$
and the index is $i_f(x) = 1-|S^-(x)| = 1-s_f(x)/2$, where $s_f(x)$ is the number of sign changes of $f$ 
on the cyclic graph $S(x)$. A small computation shows that in that particular case, 
the integral over all Morse functions $\E[s(x)] = |S(x)|/3$ and 
that $\E[1-s(x)/2] = 1-|S(x)|/6$ agrees again with curvature. This special case of the
index expectation result led us to the more general result proven here.  \\

To keep this paper self contained, the proofs of the Gauss-Bonnet and Poincar\'e-Hopf results
are attached in an appendix. 
These general results become more geometric when dealing with graphs which are triangularizations of manifolds.
In that case, Gauss-Bonnet is a discretization of Gauss-Bonnet-Chern and Poincar\'e-Hopf is a discretisation
of the analogue classical result in the case of gradient fields. In the continuum, for Riemannian manifolds, 
Euler curvature is only defined for even dimensional manifolds. 
This paper is a step towards proving that for odd dimensional graphs the curvature
is always zero, something we know only in dimensions $1$ and $3$ so far. In an upcoming paper, using further developed
techniques initiated here but using geometric assumptions on graphs like that unit spheres share properties of the continuum
unit spheres in $d$ dimensions, we will prove that for odd dimensional geometric graphs, the symmetric index 
$j_f(x)$ is zero everywhere.  This matches the continuum case, where for Morse functions $f$ the Brouwer index at a critical point 
is $i_f(x) = (-1)^{m(x)}$ where $m(x)$ is the Morse index, the number of negative eigenvalues of the Hessian matrix 
$H(x)$ at the critical point $x$. In odd dimensions, of course
$j_f(x) = [i_f(x) + i_{-f}(x)]/2=0$ at every critical point implying immediately Poincar\'e's result that 
odd-dimensional manifolds have zero Euler characteristic. We still are in search for
continuum analogue of Theorem~(\ref{mainresult}). The technical difficulty is to find a natural probability 
space of $C^2$ Morse functions on a compact Riemannian manifold.
This is not a problem in the case of graph as we will see in the next section.

\section{Index expectation}

We first define the probability space of injective functions on the vertex set $V$ of the graph $G$.
Denote by $n$ the order of the graph, the number of vertices in $V$.

\begin{defn}
Let $\Omega \subset [-1,1]^n$ be the subset of all injective 
functions on $V$ taking values in $[-1,1]$. This is a $n$-dimensional Lebesgue space.
We assume that $\Omega$ is equipped with the product Lebesgue measure ${\rm P}$. 
This means that ${\rm P}[ \{ f \; | \; f(x) \in [a,b] \} ] = (b-a)/2$ if $-1 \leq a <b \leq 1$ and that the random variables 
$X_v(f) = f(v)$ giving the function values on the vertices $v \in V$ are independent and identically distributed. 
The injective functions are the complement of a union $\Sigma$ of hyper surfaces in $[-1,1]^n$ and have full measure.
Denote by ${\rm E}[i_f(x)]$ the {\bf expectation} of the index $i_f(x)$ at the vertex $x \in V$
of $f \in \Omega$ in this probability space $(\Omega,{\rm P})$.
\end{defn}

In order to prove the main theorem, we need an excursion to percolation theory 
(We do not look at classical problems but for background, see \cite{Grimmet,BollobasRiordan}), 
in particular site percolation, where vertices of a graph $S$ are killed with a certain probability:
given a background graph $S$ and a fixed graph $H$, denote by $v_H$ the number of times
the graph appears embedded in $S$. Now switch off vertices and edges connecting them in $S$ independently from 
each other with probability $p$. Call $v_H^p$ the expected number of graphs which appear now.
It depends on $S$ and $p$ as well as $H$. We will see however that $\int_0^1 v_H^p\; dp$
only depends on $H$. In our case, we need the situation when $H$ is the
$k$-dimensional simplex, the complete graph with $k+1$ vertices. \\

Denote by $V_k(x)$ the number of $H=K_{k+1}$ subgraphs in the sphere $S(x)$ and by $V_k^-(x)$ the number of $K_{k+1}$ 
subgraphs in the exit set $S^-(x)$, the subgraph of $S(x)$ generated by vertices $y$ where $f(y)<f(x)$. 
Let $v_H$ denote the number of simplices $K_{k+1}$ which appear as subgraphs in $S$. We can look at the 
Erd\"os-Renyi probability space \cite{bollobas}
of all subgraphs of $S$, where each vertex is included with probability $p$ and the subgraph is the graph generated
by these vertices. Let $v_H^p$ the expected number of $k$-dimensional simplices in the decimated subgraph of $S$. 
Obviously $v_H^p \leq v_k$, but how much?  
Computing the expectation $v_k^p$ of the survival rate depends on $S$ and $p$ But if $p$ is chosen randomly too 
at first and each vertex is deleted with probability $p$, the survival rate only depends on the order
of the clique and not on the graph:

\begin{propo}[Clique survival for site percolation]
$$  \int_0^1 \frac{\E_p[v_H^p]}{v_k} \; dp = \frac{1}{k+2} \; . $$
\end{propo}
\begin{proof}
The result is true if all the $K_{k+1}$ graphs in $S$ are disjoint because the survival of a single isolated
simplex with $k+1$ vertices is $\int_0^1 p^{k+1} \; dp = 1/(k+2)$.  \\
To prove the result in general, we decorrelate the situation by splitting 
vertices: pick a vertex $v$ where at least two such $K_{k+1}$ subgraphs $H_1,H_2$ intersect. 
Replace $v$ with $2$ vertices$v_1,v_2$ and place the edges to $H_1$ with $v_1$ and edges to 
$H_2$ with $v_2$. Distribute the other edges originally intersecting with $v$ arbitrarily with $v_1$ or $v_2$. 
To show that $\int_0^1 E_p[v_k^p] \; dp$ does not change when passing to the larger probability
space, we compare the case before and after splitting: 
before splitting, $v$ appears with probability $p$ and contributes to $H_1$ and $H_2$. This gives $\int_0^1 2p^2 \; dp = 1$. 
After splitting, both $v_i$ appear together with probability $p^2$, exactly one 
appears with probability $2 p(1-p)$ and none appears with probability $(1-p)^2$.
This also leads to a contribution $2 \int_0^1 p^2 \; dp + 1 \int_0^1 2 p(1-p) \; dp + 0 \int_0^1 (1-p)^2 \; dp = 1$.
We repeat like this with other intersection points of $H_1$ and $H_2$.
After all the correlations between $H_1,H_2$ are unlocked, we have a situation where the two simplices are independent
and where the expectation value is the same as before. Now proceed with any other pair of simplices $K_{k+1}$ 
in the same way.  
\end{proof}

The result means for $k=0$ that half of the points survive and for $k=1$ that $1/3$ of all 
edges are expected to survive. We ran Monte Carlo simulations with random host graphs $G$ which 
is fixed over the experiment, where each vertex is knocked off with probability $p$. 
Applying $m$ such disaster experiments, each time starting fresh with the same $G$ and then repeating the
experiments for various $p$ and averaging over disaster severeness $p$ confirms the result 
remarkably well with errors of the order $1/m$ and smaller simplices $K_{k+1}$ like  $k=2,3,4$. \\

This result is actually more general. The clique graph $K_{k+1}$ can be an arbitrary {\bf pattern graph} 
$H$. The result holds both for site and bond percolation situations.
For {\bf site percolation catastrophes}, the nodes are killed with probability $p$, 
in {\bf bond percolation catastrophes}, the edges are broken with probability $p$. 
For an arbitrary background host graph $S$ and any fixed pattern graph $H$, the expected {\bf decimation rate} 
for the number of patterns $H$ occurring in $S$ is $1/({\rm ord}(H)+1)$ for site disasters and $1/({\bf size}(H)+1)$ 
for bond disasters. These network stability results are remarkably universal: they are 
independent of the background graph $S$. They can serve as "rules of thumb" if one has no a priory idea about the 
disaster strength $p$. 

\begin{coro}[Averaging equation]
For every vertex $x \in V$ and all $k \geq 0$,
$$  \E[V_{k}^-(x)] = \frac{V_{k}(x)}{k+2}  \; . $$
\end{coro}
\begin{proof}
Look at a central vertex $x$ connected to other vertices $z_i$. We want to apply the previous lemma 
for $S=S(x)$. Because $f$ takes values in $[-1,1]$, we can
assume $f(x)=-1+2p$ with $p \in [0,1]$. Having $f(x)$ fixed like that, we get a random site
percolation problem in the sphere $S(x)$, where each vertex $y \in S(x)$ appears 
with probability $p$ independently of each other. 
The expected number $V_k^-(x)$ of $k$-dimensional simplices $K_{k+1}$ divided by the number $V_k(x)$ 
of simplices in $S(x)$ is by the previous lemma equal to $1/(k+2)$ after we integrate over $p$. 
\end{proof}

For $k=0$, we have $\E[V_0^+(x)]=\E[V_0^-(x)]$ because the probability space is invariant under the involution $f \to -f$
so that $\E[V_0^-(x)] = V_0(x)/2$ follows.  For $k=1$, the averaging equations are $\E[V_1^-(x)] = V_1(x)/3$ and
$\E[V_1^-(x)]=\E[V_1^+(x)]=\E[W_1(x)]$, where $W_1(x)$
is the set of vertices connecting vertices from $V_1^+(x)$ to $V_1^-(x)$. \\

Here is the main result: 

\begin{thm}[Index expectation is curvature]
\label{mainresult}
For every vertex $x$, the expectation of $i_f(x)$ is $K(x)$:
$$  \E[i_f(x)] = K(x) \; . $$
\end{thm}

\begin{proof}
We the averaging equation where $k$ is replaced by $k-1$
$$ \frac{V_{k-1}(x)}{k+1} = \E[ V_{k-1}^-(x)] $$
to see
\begin{eqnarray*}
   \E[ 1-\chi(S^-(x)) ] &=& 1-\sum_{k=0}^{\infty}  (-1)^k \E[V_k^-(x)] \\
                       &=& 1-\sum_{k=0}^{\infty} (-1)^k \frac{V_{k}(x)}{(k+2)} \\
                       &=& 1+\sum_{k=1}^{\infty} (-1)^k \frac{V_{k-1}(x)}{(k+1)} \\
                       &=& \sum_{k=0}^{\infty} (-1)^k \frac{V_{k-1}(x)}{(k+1)} \\
                       &=& K(x)   \; .
\end{eqnarray*}
\end{proof}

{\bf Remark.} This gives a new proof of the discrete Gauss-Bonnet result 
$$   \sum_{x \in V} K(x) = \chi(G) \;  $$
from 
$$   \sum_{x \in V} i_f(x) = \chi(G) \;  $$
simply by taking expectation.  But unlike in the continuum, where Gauss-Bonnet-Chern is
more difficult to prove (see e.g \cite{BergerPanorama,Cycon}), the discrete Gauss-Bonnet is easy to prove
directly. As demonstrated in the Appendix, Gauss-Bonnet for graphs is even more direct than Poincar\'e-Hopf.
We expect that for compact Riemannian manifolds, a new probabilistic link between Poincar\'e-Hopf and Gauss-Bonnet 
will allow to simplify the proof of the later considerably in higher dimensions. In the continuum, 
Poincar\'e-Hopf is orders of magnitudes less complex than Gauss-Bonnet-Chern 
because it is part of {\bf differential topology}, not needing
any Riemannian metric while Gauss-Bonnet is part of {\bf differential geometry} which uses more 
structure on the manifold $M$. It is the probability space on Morse function which will add part of the Riemannian
structure on $M$, enough to get curvature. In the continuum, there are various probability spaces which are good 
candidates to represent curvature as index expectation. They all appear to work for compact two-dimensional surfaces.

\section*{Appendix}

Here are the proofs of Gauss-Bonnet \cite{cherngaussbonnet} and Poincar\'e-Hopf \cite{poincarehopf} for simple 
graphs $G=(V,E)$ with consolidated notation. For Mathematica code, see \cite{KnillWolframDemo1,KnillWolframDemo2}.
More Mathematica code illustrating all the probabilistic aspects proven in this paper and
\cite{randomgraph} will become demonstrations too. \\

The first lemma generalizes Euler's handshaking lemma $\sum_{x \in V} V_0(x)=2 v_1$:

\begin{lemma}[Transfer equations]
\label{transferequations}
$\sum_{x \in V} V_{k-1}(x) = (k+1) v_k$.
\end{lemma}

\begin{proof}
We can interpret $V_{k-1}(x)$ as the {\bf $k$-degree} of a vertex $v$, 
the number of $k$-simplices $K_{k+1}$ which contain $v$. 
The sum over all $k$-degrees $\sum_k {\rm deg}_k(x)$ is $k+1$ times the number $v_k$ of 
$k$-simplices $K_{k+1}$ in $G$.  \\
\end{proof}

Here is an other more pictorial proof of the transfer equations: 
draw and count handshakes from every vertex to every center of any $k$-simplex in two 
different ways. A first count sums up all connections leading to
a given vertex, summing then over all vertices leading to $\sum_{x \in V} V_{k-1}(x)$.
A second count is obtained from the fact that every simplex has $k+1$ hands reaching out 
and then sum over the simplices gives $(k+1) v_k$ handshakes.

\begin{thm}[Gauss-Bonnet]
$\sum_{x \in V} K(x) = \chi(G)$.
\label{theorem1}
\end{thm}

\begin{proof}
By definition of curvature, we have
$$  \sum_{x \in V} K(x) = \sum_{x \in V} \sum_{k=0}^{\infty} (-1)^k \frac{V_{k-1}(x)}{k+1} \; . $$
Since the sums are finite, we can change the order of summation. Using the transfer equations 
(\ref{transferequations}), we get
$$  \sum_{x \in V} K(x) = \sum_{k=0}^{\infty} \sum_{x \in V} (-1)^k \frac{V_{k-1}(x)}{k+1}  
                        = \sum_{k=0}^{\infty} (-1)^k v_k = \chi(G) \;. $$
\end{proof}

Given an injective function $f$ on $V$
and a vertex $x \in V$, we can look at the set $W_k(x)$ of all $k$ simplices in the sphere $S(x)$
for which at least one vertex $y$ satisfies $f(y)<f(x)$ and an other vertex $z$ satisfies $f(z)>f(x)$.

\begin{lemma}[Intermediate equations]
$\sum_{x \in V} W_k(x) = k v_{k+1}$
\label{intermediateequations}
\end{lemma}

\begin{proof}
For each of the $v_{k+1}$ simplices $K_{k+2}$ in $G$, 
there are $k$ vertices $x$ which have neighbors in $K_{k+2}$ with both larger and 
smaller values. For each of these $k$ vertices $x$, we can look
at the unit sphere $S(x)$ of $v$. The simplex $K_{k+2}$ defines a 
$k$-dimensional simplex $K_{k+1}$ in that unit sphere. Each of them 
adds to the sum $\sum_{x \in V} W_k(x)$ which consequently is equal to 
$k v_{k+1}$.
\end{proof}

\begin{lemma}[Index stability]
The index sum $\sum_{x \in V} i_f(x)$ is independent of $f$. 
\label{indexstability}
\end{lemma}
\begin{proof}
The proof is a deformation argument. Fix a vertex $x$ and change the value of the function 
$f(x)$ such that a single neighboring point $y \in S(x)$, the value $f(y)-f(x)$ changes sign 
during the deformation. Without loss of generality we can 
assume that the value of $f$ at $x$ has been positive initially and gets negative. 
Now $S^-(x)$ has gained a point $y$ and $S^-(y)$ has lost a point.  \\
To see that $\chi(S^-(x)) + \chi(S^-(y))$ stays constant, 
we check this each individual simplex level and show $V_k^+(x)+ V_k^-(t)$ stays constant,
where $V_k^{\pm}(x)$ denotes the number $K_{k+1}$ subgraphs of 
$S(x)$ which connect points within $S(x)^{\pm}$.
Since $i(x) = 1- \sum_k (-1)^k V_k^-(x)$, the lemma is proven if $V_k^-(x) + V_k^-(x)$ stays constant 
under the deformation. Let $U_k(x)$ denote the number of $K_{k+1}$ subgraphs of
$S(x)$ which contain $y$. Similarly, let $U_k(y)$ the number of $K_{k+1}$ subgraphs of
$S(y)$ which do not contain $x$ but are subgraphs of $S^-(y)$ with $x$.
The sum of $K_{k+1}$ graphs of $S^-(x)$ changes by $U_k(y)-U_k(x)$.
When summing this over all vertex pairs $x,y$, we get zero.
\end{proof}

\begin{thm}[Poincar\'e-Hopf]
$\sum_{x \in V} i_f(x) = \chi(G)$
\end{thm}

\begin{proof}
The number of $k$-simplices $V_k^-(x)$ in the exit set $S^-(x)$
and the number of $k$-simplices $V_k^+(x)$ in the entrance set $S^+(x)$ are complemented 
within $S(x)$ by the number $W_k(x)$ of $k$ simplices which contain both vertices from $S^-(x)$ 
and $S^+(x)$. By definition, $V_k(x) = W_k(x) + V_k^+(x) + V_k^-(x)$.
By the index stability lemma (\ref{indexstability}), the index $i_f(x)$ is the same for all 
injective functions $f:V \to \mathbf{R}$. Let  $\chi'(G) = \sum_{x \in V} i_f(x)$. 
Because replacing $f$ and $-f$ switches $S^+$ with $S^-$ and the sum is the same, we
can prove $2 v_0 - \sum_{x \in V} \chi(S^+(x)) + \chi(S^-(x)) = 2 \chi'(G)$ instead.
The transfer equations Lemma~(\ref{transferequations}) and intermediate equations Lemma~(\ref{intermediateequations}) give
\begin{eqnarray*}
 2 \chi'(G) &=& 2v_0 + \sum_{k=0}^{\infty} (-1)^k \sum_{x \in V} (V_k^-(x) + V_k^+(x)) \\
            &=& 2v_0 + \sum_{k=0}^{\infty} (-1)^k \sum_{x \in V} (V_k(x)   - W_k(x)  ) \\
            &=& 2v_0 + \sum_{k=0}^{\infty} (-1)^k [ (k+2) v_{k+1} - k v_{k+1}]  \\
            &=& 2v_0 + \sum_{k=1}^{\infty} (-1)^k 2 v_{k} = 2 \chi(G) \; .
\end{eqnarray*}
\end{proof}

\vspace{12pt}
\bibliographystyle{plain}

\begin{thebibliography}{1}

\bibitem{bollobas}
B.Bollob{\'a}s.
\newblock {\em Random graphs}, volume~73 of {\em Cambridge Studies in Advanced
  Mathematics}.
\newblock Cambridge University Press, Cambridge, second edition, 2001.

\bibitem{BergerPanorama}
M.~Berger.
\newblock {\em A Panoramic View of Riemannian Geometry}.
\newblock Springer Verlag, Berlin, 2003.

\bibitem{BollobasRiordan}
B.~Bollobas and O.~Riordan.
\newblock {\em Percolation}.
\newblock Cambridge University Press, 2006.

\bibitem{Cycon}
H.L. Cycon, R.G.Froese, W.Kirsch, and B.Simon.
\newblock {\em {Schr\"odinger} Operators---with Application to Quantum
  Mechanics and Global Geometry}.
\newblock Springer-Verlag, 1987.

\bibitem{Grimmet}
G.~Grimmet.
\newblock {\em Percolation}.
\newblock Springer Verlag, 1989.

\bibitem{cherngaussbonnet}
O.~Knill.
\newblock A graph theoretical {Gauss-Bonnet-Chern} theorem.\\
\newblock http://arxiv.org/abs/1111.5395, 2011.

\bibitem{poincarehopf}
O.~Knill.
\newblock A graph theoretical {Poincar\'e-Hopf} theorem. \\
\newblock http://arxiv.org/abs/1201.1162, 2012.

\bibitem{randomgraph}
O.~Knill.
\newblock The dimension and {Euler} characteristic of random graphs. \\
\newblock http://arxiv.org/abs/1112.5749, 2011.

\bibitem{KnillWolframDemo1}
O.~Knill.
\newblock Dimension and {E}uler characteristics of graphs. \hfill\\
\newblock {\em Wolfram Demonstrations Project}, Jan 31, 2012. \\
\begin{tiny}http://demonstrations.wolfram.com/DimensionAndEulerCharacteristicsOfGraphs,\end{tiny}

\bibitem{KnillWolframDemo2}
O.~Knill.
\newblock {G}auss-{B}onnet and {P}oincare-{H}opf for graphs. \hfill\\
\newblock {\em Wolfram Demonstrations Project}, Jan 31, 2012. \\
\begin{tiny}http://demonstrations.wolfram.com/GaussBonnetAndPoincareHopfForGraphs,\end{tiny} 

\end{thebibliography}

\end{document}